\documentclass[a4paper,11pt]{article}

\usepackage[T1]{fontenc} 
\usepackage[utf8]{inputenc}
\usepackage[ngerman,american]{babel} 

\usepackage{amsmath}
\usepackage{amsfonts}
\usepackage{amssymb}
\usepackage{graphicx}
\usepackage{color}
\usepackage{tikz} 

\usepackage{xspace}

\usepackage{fullpage}
\usepackage{amsthm}
\theoremstyle{plain}
\newtheorem{theorem}{Theorem}
\newtheorem{lemma}[theorem]{Lemma}

\newtheorem{corollary}[theorem]{Corollary}
\theoremstyle{definition}
\newtheorem{definition}[theorem]{Definition}

\newtheorem{conjecture}{Conjecture}

\usepackage[colorlinks]{hyperref}

\def\nicebreak{\vskip0pt plus50pt\penalty-300\vskip0pt plus-50pt }

\definecolor{lightblue}{rgb}{0.5,0.5,1.0}
\definecolor{darkred}{rgb}{0.5,0,0}
\definecolor{darkgreen}{rgb}{0,0.5,0}
\definecolor{darkblue}{rgb}{0,0,0.5}

 \hypersetup{colorlinks
 ,linkcolor=darkred
 ,filecolor=darkgreen
 ,urlcolor=darkred
 ,citecolor=darkblue}

\newcommand{\CClassNP}{\textup{NP}\xspace}

\newcommand{\NPhard}{\CClassNP-hard\xspace}

\newcommand{\dunion}{\mathbin{\dot{\cup}}}

\newcommand{\isg}{\boldsymbol{H}}

\DeclareMathOperator{\defect}{def}
\DeclareMathOperator{\hole}{hole}
\DeclareMathOperator{\cp}{cn}

\title{Competition Numbers, Quasi-Line Graphs and Holes}
\author{Brendan D.~McKay, Pascal Schweitzer and Patrick Schweitzer\thanks{This
work is supported 
by the Australian Research Council, the Fonds National de la Recherche, Luxembourg, and co-funded under the Marie Curie 
Actions of the European Commission (FP7-COFUND). } \\[2ex]
Research School of Computer Science\\
The Australian National University\\ 
Canberra, ACT 0200, Australia\\
{\tt bdm@cs.anu.edu.au, Pascal.Schweitzer@anu.edu.au} \\[1ex]
University of Luxembourg\\ 
Interdisciplinary Centre for Security, Reliability and Trust\\
6, rue Richard Coudenhove-Kalergi,
L-1359 Luxembourg\\
{\tt Patrick.Schweitzer@uni.lu}
}

\begin{document}

\maketitle

\begin{abstract}

The competition graph of an acyclic directed graph~$D$ is the 
undirected graph on the same vertex set as~$D$ in which two distinct
vertices are adjacent if they have a common out-neighbor in~$D$. 
The competition number of an undirected graph~$G$ is the least number of isolated vertices
that have to be added to~$G$ to make it the 
competition graph of an acyclic directed graph. 
We resolve two conjectures concerning competition graphs. First we prove a 
conjecture of Opsut by showing that the 
competition number of every quasi-line graph is at most~$2$. Recall that a 
quasi-line graph, also called a locally co-bipartite graph,
is a graph for which the neighborhood of every vertex can be partitioned into at most two 
cliques. To prove this conjecture we devise an alternative  
characterization of quasi-line graphs to the one by Chudnovsky and Seymour.
Second, we prove a conjecture of Kim by showing that the competition number of any graph is at most one greater than the number of holes in the graph.
Our methods also allow us to prove a strengthened form of this conjecture recently proposed by Kim, Lee, Park and Sano, showing that the competition number of any graph is at most one greater than the dimension of the subspace of the cycle space spanned by the holes.
\end{abstract}

\section{Introduction} \label{sec:introduction}

The \emph{competition graph} of an acyclic directed graph~$D$ is the 
undirected graph on the same vertex set as~$D$ in which two distinct
vertices~$u$ and~$v$ are 
adjacent if there is a vertex~$w$ such that~$(u,w)$ and~$(v,w)$ are arcs
in~$D$. That is, two vertices in the competition graph are adjacent if they
have a common out-neighbor in~$D$.
Competition graphs were introduced by Cohen~\cite{Cohen} in the context
of food webs, where adjacency of two vertices models the fact that they share
common prey and thus compete for food.

Roberts~\cite{Robertsfoodwebs} observed that by adding a sufficient number of
isolated vertices, every undirected graph~$G$ can be turned into the
competition graph of some acyclic directed graph~$D$. Quantifying this, for an
undirected graph~$G$, the \emph{competition number} of~$G$, 
denoted~$\cp(G)$, is the least number of isolated vertices
that have to be added to~$G$ to make it the 
competition graph of an acyclic directed graph~$D$. 

Opsut~\cite{MR679638} showed that computing the competition number is an
\NPhard problem. He also showed that the competition
number of a line graph is at most~$2$. 
He then conjectured that the bound also holds
for quasi-line graphs, i.e., for graphs in which the
neighborhood of every vertex can be partitioned into at most two cliques.

\begin{conjecture}[Opsut~\cite{MR679638}]\label{conj:opsut}
If~$G$ is a quasi-line graph, then~$G$ has competition number at most~$2$.
\end{conjecture}

Another way to bound the competition number is to consider the number of holes 
a graph contains.
Recall that a hole is an induced cycle of length at least~$4$. In this context
Kim~\cite{MR2176262} conjectured the following:

\begin{conjecture}[Kim~\cite{MR2176262}]\label{conj:kim}
If a graph has at most~$k$ holes, then it has competition number at most~$k+1$.
\end{conjecture}

These conjectures have previously been proven for various graph classes, as summarized at the end of this section. However, both conjectures have remained open until now. 

\paragraph{Our results.} In this paper we prove Conjectures~\ref{conj:opsut}
and~\ref{conj:kim}. To prove Conjecture~\ref{conj:opsut}, we first present an
alternative to Chudnovsky and 
Seymour's~\cite{DBLP:conf/bcc/ChudnovskyS05} structure theorem for quasi-line graphs. In the course of developing a
structure theorem for the more general class of claw-free graphs, they show 
that
every connected quasi-line graph is a fuzzy circular interval graph, or a
composition of fuzzy linear interval strips.

We will show that Chudnovsky and Seymour's theorem implies the following
simpler characterization, proven in Section~\ref{char:sec:quasi-line}.

\begin{theorem}\label{thm:quasi:line:char}
A finite graph~$G$ is a quasi-line graph, if and only if there exists a finite 
graph~$H$, a map~$\phi\colon V(G)\rightarrow V(H)$, and a set of distinct
connected 
subtrees~$\isg:=\{T_1,\ldots,T_t\}$ of~$H$ such that the 
following properties hold: 
\begin{enumerate}
\item If two vertices~$v$ and~$v'$ of~$G$ are adjacent, then there is a 
tree~$T_i \in \isg$ with~$\phi(v),\phi(v')\in
V(T_i)$.\label{prop:if:adj:subgraph}
\item If two vertices~$v$ and~$v'$ of~$G$ are not adjacent 
and~$\phi(v),\phi(v')\in V(T_i)$ for some~$T_i \in \isg$, then~$T_i$ is a path with distinct
endpoints~$\phi(v)$ and~$\phi(v')$. \label{prop:fuzzy:means:pathends}
\item If~$v_1$ is not adjacent to~$v_1'$ and~$v_2$ is not adjacent to~$v_2'$,
but~$\phi(v_1) = \phi(v_2)$,~$\phi(v_1),\phi(v_1')\in V(T_i)$
and~$\phi(v_2),\phi(v_2')\in V(T_j)$ for trees~$T_i$ and~$T_j$ in~$\isg$, then~$i = j$.
\label{prop:paired:fuzzyness} 
\item Every vertex~$u\in V(H)$ of degree at least~$3$ is contained in at most 
one tree~$T_i \in \isg$.\label{prop:only:one:H:if:deg:greater:2}
\end{enumerate}
\end{theorem}

We call any triple~$(H, \phi, \isg)$, consisting of the graph~$H$ 
together with the map~$\phi$ and the family~$\isg$ of subtrees, \emph{a fuzzy 
reconstruction} of the quasi-line graph~$G$.
We call a pair of distinct vertices~$\{v,v'\} 
\subseteq V(G)$
\emph{fuzzy} if the pair satisfies the assumption of Property~\ref{prop:fuzzy:means:pathends}, i.e, if the two vertices are non-adjacent and there is a 
tree~$T\in \isg$ with~$\phi(v),\phi(v')\in V(T)$.
In this case~$T$ is a 
path with endpoints~$\phi(v)$ and~$\phi(v')$. Figure~\ref{figure1} illustrates a reconstruction of a graph~$G$ and its fuzzy vertex pairs.

\begin{figure}[h!t]
\centering
\unitlength 0.9cm
\begin{picture}(16,18.5)(0,0)
\put(1.5,0){\includegraphics[scale=0.81]{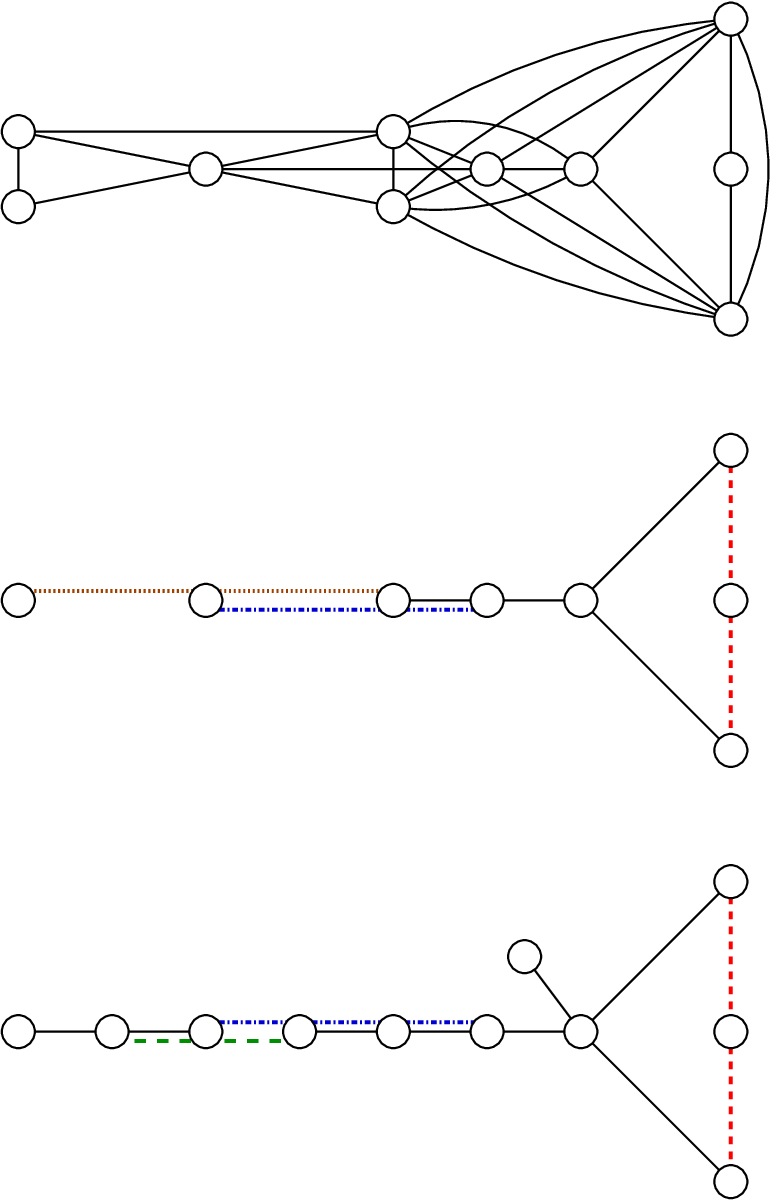}}
\put(7,13.3){\large$G$}
\put(6.1,7.2){\large$(H,\phi,\isg)$}
\put(5.9,0.5){\large$(H',\phi',\isg')$}
\put(1,12){%
\put(0.7,2.4){1}
\put(0.65,4.7){2}
\put(3.55,3.0){3}
\put(6.35,2.4){4}
\put(6.4,4.7){5}
\put(8.0,4.0){6}
\put(9.7,3.6){7}
\put(12.1,1.3){8}
\put(11.0,3.6){9}
\put(12.0,5.9){10}
}
\put(0.7,5.5){%
\put(0.1,4.2){$\phi(1){=}\phi(2)$}
\put(3.6,4.2){$\phi(3)$}
\put(5.8,4.2){$\phi(4){=}\phi(5)$}
\put(7.9,2.9){$\phi(6)$}
\put(9.0,4.1){$\phi(7)$}
\put(12.35,1.25){$\phi(8)$}
\put(12.35,3.55){$\phi(9)$}
\put(12.35,5.9){$\phi(10)$}
}
\put(0.9,-1){%
\put(0.5,4.1){$\phi'(1)$}
\put(1.9,4.1){$\phi'(2)$}
\put(3.3,4.1){$\phi'(3)$}
\put(6.2,4.1){$\phi'(4)$}
\put(4.75,4.1){$\phi'(5)$}
\put(7.6,2.8){$\phi'(6)$}
\put(8.2,5.2){$\phi'(7)$}
\put(12.15,1.2){$\phi'(8)$}
\put(12.15,3.45){$\phi'(9)$}
\put(12.15,5.75){$\phi'(10)$}
}

\end{picture}
\caption{Example of a quasi-line graph $G$ together with a reconstruction
   $(H,\phi,\isg)$ that uses four subtrees.  The fuzzy pairs are $\{1,4\}$, $\{1,5\}$,
   and $\{2,4\}$.  A second example of a reconstruction of
   $G$ is $(H',\phi',\isg')$. It uses five subtrees, has no fuzzy pairs and is pleasant (as defined in Section~\ref{char:sec:quasi-line}).\label{figure1}}
\end{figure}
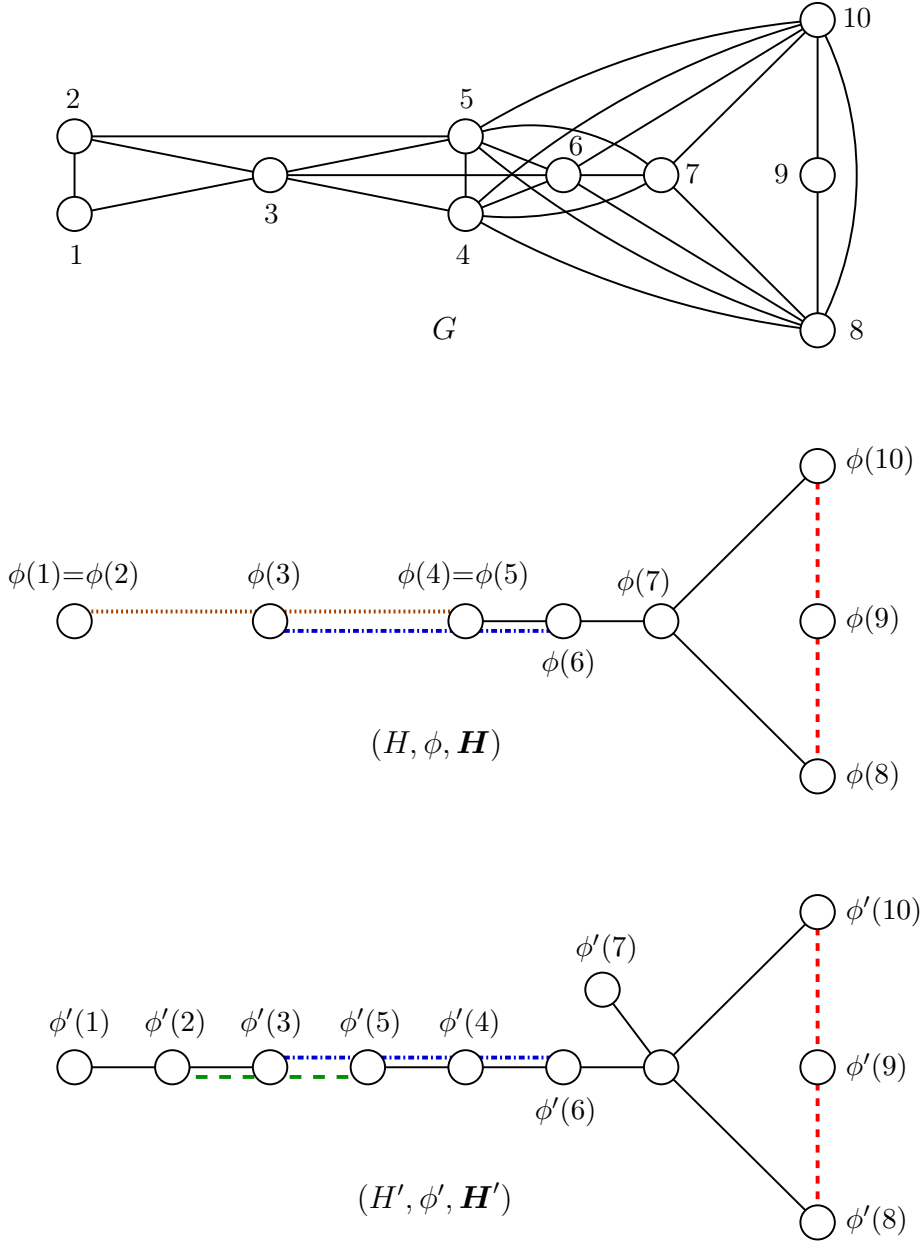

Note that it is not necessarily
possible to deduce~$G$ from the triple~$(H, \phi, \isg)$.
To easily apply fuzzy reconstructions in our proofs, we also show in Section~\ref{char:sec:quasi-line} that quasi-line graphs have fuzzy reconstructions with various additional properties. We call these pleasant reconstructions.
Using fuzzy reconstructions we prove the conjecture of Opsut in
Section~\ref{sec:quasi-line}.

In Section~\ref{sec:competition_number_bound} we then prove
Conjecture~\ref{conj:kim}. Our proof method even allows us to prove a stronger 
form of the conjecture (Theorem~\ref{thm:dimension:hole:space}) that was recently
considered in~\cite{ComHoleSpace}.

\paragraph{Related work.}

Several partial results, proving Opsut's conjecture for certain graph classes, 
have been previously obtained. In particular it has been proved for 
quasi-line graphs for which there are coverings of the neighborhood of all adjacent vertices that are compatible in a certain sense~\cite{MR1041628}. This result was extended to all non-critical quasi-line graphs~\cite{MR1206561}. In this context, critical graphs are all quasi-line graphs in which every clique~$C$ has a vertex whose neighbors outside~$C$ cannot be covered by a single clique.
Subsequently the conjecture has also been proven for bubble-free 
series-parallel graphs~\cite{wang95}.
Opsut's original paper~\cite{MR679638}, 
which contains Conjecture~\ref{conj:opsut}, shows an upper and a lower bound 
in terms of the clique edge cover number. The clique edge cover number of a 
graph is the least number of cliques that together cover every edge. Assuming 
his conjecture, one can use his bounds to deduce that every quasi-line graph 
on~$n$ vertices has a clique edge cover of size at most~$n$. 
Chen, Jacobson, K\`{e}zdy, Lehel, Scheinerman and Wang~\cite{Chen200017} 
proved that indeed every quasi-line graph has such a clique edge covering. 
Thus, by proving Conjecture~\ref{conj:opsut}, we obtain an alternative proof 
for their result.

Concerning Conjecture~\ref{conj:kim}, there are numerous papers that contain
partial results. For graphs without holes the conjecture was already proven in 
the original paper that introduced competition numbers~\cite{Robertsfoodwebs}. 
This result was 
extended to graphs with at most one hole~\cite{MR2181040} and then to graphs 
with at most two holes~\cite{MR2656247,boli}. The conjecture was also proven for 
graphs in which the intersection of holes is restricted in a certain way. More 
specifically the conjecture was proven for graphs in which holes either 
share at most one vertex, or share an edge and contain at least~$5$ 
vertices~\cite{MR2510214}. It was also proven for graphs with mutually 
edge disjoint holes~\cite{MR2651995} and for graphs in which each hole has an 
edge that is not contained in any other hole~\cite{MR2563281}.
Theorem~\ref{thm:dimension:hole:space}, the stronger form of Conjecture~\ref{conj:kim}, was recently shown to hold for most graph classes for which 
Conjecture~\ref{conj:kim} was known to hold~\cite{ComHoleSpace}.

\paragraph{Notation.}

Throughout the paper we use finite simple graphs that may either have directed edges, which we call arcs, or undirected edges.
By~$V(G)$ and~$E(G)$ we denote vertices and edges of a graph~$G$, respectively. For a subset of vertices~$V'\subseteq V(G)$ we let~$G[V']$ be the subgraph 
of~$G$ induced by the vertices in~$V'$. For a vertex~$v$ in an undirected graph, we let~$N_G(v) := \{u\mid \{u,v \}\in E(G)\}$  be
the \emph{open neighborhood} of~$v$ in~$G$. 
We omit the index~$G$ if the graph is 
apparent from the context. 
A \emph{clique} is a set of vertices that induces a complete graph.
Recall that a 
\emph{simplicial vertex} is a vertex whose neighborhood forms a clique.
A set of cliques~$C_1,\ldots,C_t$ \emph{covers} a set of edges~$E$ if for every edge~$e\in E$ there is an~$i\in\{1,\ldots,t\}$ such that~$e$ is an edge in the graph induced by~$C_i$.

For~$S\subseteq V(G)$, we define~$G-S := 
G[V(G)\setminus S]$ to be the graph obtained by removing vertices in~$S$ and
all edges incident with a vertex in~$S$. When 
removing a single vertex~$v$ we write~$G-v := G-\{v\}$.
A \emph{vertex separator} in a graph~$G$ is a (possibly empty) subset of the 
vertices~$S\subseteq V(G)$ such that~$G-S$ is a disconnected graph. 

\section{A characterization of quasi-line graphs} \label{char:sec:quasi-line}

In this section we prove Theorem~\ref{thm:quasi:line:char}, which says that
a graph has a fuzzy reconstruction if and only if it is a quasi-line graph.
When referring to a fuzzy reconstruction, in the rest of the paper, we 
frequently omit the word fuzzy.
When dealing with a reconstruction~$(H, \phi, \isg)$, to improve readability, 
we use the letters~$u$,~$v$ and~$w$ to denote vertices of the graph~$G$ and we 
use the letters~$h$ and~$g$ or expressions of the form~$\phi(v)$ to denote 
vertices of the graph~$H$. 
With the notion of fuzzy vertex pairs, we can reformulate Property~\ref{prop:paired:fuzzyness}. More precisely, if Properties~\ref{prop:if:adj:subgraph},~\ref{prop:fuzzy:means:pathends}
and~\ref{prop:only:one:H:if:deg:greater:2} hold, then 
Property~\ref{prop:paired:fuzzyness} is equivalent to the
following statement:

\newcounter{listcounter}
\begin{list}{$\arabic{listcounter}'$.}{\usecounter{listcounter}}
\setcounter{listcounter}{2} 
\item If~$\{v_1,v_1'\}\subseteq V(G)$ is a fuzzy vertex pair, then there is 
exactly one tree~$T \in \isg$ that contains~$\phi(v_1)$ and~$\phi(v_1')$. 
Furthermore, if~$\{v_2,v_2'\}\subseteq V(G)$ is also a fuzzy vertex pair, then 
the sets~$\{\phi(v_1),\phi(v_1')\}$ and~$\{\phi(v_2),\phi(v_2')\}$ are either 
equal or disjoint.
\label{prop:three:prime}

\end{list}

We recall several definitions from~\cite{DBLP:conf/bcc/ChudnovskyS05} (fixing the minor typo~$\{u,v\}$ to~$\{\phi(u),\phi(v)\}$). 

\nicebreak

\begin{definition}[Chudnovsky, Seymour \cite{DBLP:conf/bcc/ChudnovskyS05}]
A graph~$G$ is a \emph{fuzzy circular interval graph} if
\begin{itemize}
\item there is a map~$\phi$ from~$V(G)$ to a circle~$C$, and
\item there is a set of non-trivial closed intervals from~$C$, none including another, and such that no point of~$C$ is the end of more than one of 
the intervals, so that
\item for~$u,v$ in~$G$, if~$u,v$ are adjacent, then~$\{\phi(u),\phi(v)\}$ is a 
subset of one of the intervals,
and if~$u, v$ are non-adjacent, then~$\phi(u), \phi(v)$ are both ends of any 
interval that includes
both of them (and in particular, if~$\phi(u) = \phi(v)$, then~$u,v$ are 
adjacent).
\end{itemize}
\end{definition}

\begin{lemma}\label{lem:fuzzy:circ:has:resonstr}
Every fuzzy circular interval graph has a reconstruction.
\end{lemma}

\begin{proof}
Let~$\phi$,~$C$ and a set of intervals in~$C$ be given that satisfy the 
properties in the definition of a fuzzy circular interval graph~$G$.
Define~$H$ to be the graph on the vertex set~$\phi(V(G))$ in which two 
vertices are adjacent if they are consecutive on the circle~$C$.
Let~$\isg$ be the set of paths of~$H$ that are induced by the intersection 
of~$V(H)$ with one of the given intervals. 

We claim that the triple~$(H,\phi,\isg)$ is a reconstruction of~$G$. 
Properties~\ref{prop:if:adj:subgraph},~\ref{prop:fuzzy:means:pathends} 
and~\ref{prop:only:one:H:if:deg:greater:2} follow from the definition of 
fuzzy circular interval graphs. Property~\ref{prop:paired:fuzzyness} follows 
from the fact that all intervals have distinct endpoints.
\end{proof}

\begin{definition}[Chudnovsky, Seymour \cite{DBLP:conf/bcc/ChudnovskyS05}]
A graph~$G$ with two distinguished vertices~$a$ and~$b$ in~$V(G)$ is a 
\emph{fuzzy linear interval strip} if
\begin{itemize}
\item $a$ and~$b$ are simplicial,
\item there is a map~$\phi$ from~$V(G)$ to a line~$L$, 
\item there is a set of non-trivial closed intervals in~$L$, none including 
another, and such that no point of~$L$ is the end of more than one 
of the intervals, so that
\item for~$u,v$ in~$G$, if~$u,v$ are adjacent, then~$\{\phi(u),\phi(v)\}$ is a 
subset of one of the intervals,
and if~$u, v$ are non-adjacent, then~$\phi(u), \phi(v)$ are both ends of any 
interval including
both of them (and in particular, if~$\phi(u) = \phi(v)$, then~$u,v$ are 
adjacent) and
\item 
$\phi(a)$ and~$\phi(b)$ are different from~$\phi(v)$ for all other 
vertices~$v$ of~$G$. 
\end{itemize}

\end{definition}

We refer to the vertices~$a$ and~$b$ in the definition as ends of the strip, 
and denote by~$(G, a, b)$ the strip with these ends.

\begin{lemma}\label{lem:reconst:end:not:in:fuz}
If~$G$ is a fuzzy linear interval strip, then~$G$ has a 
reconstruction~$(H,\phi,\isg)$ for which~$H$ is a path and the ends do not 
have the same image under~$\phi$ as any other vertex of~$G$ and no end is part of a fuzzy pair.
\end{lemma}
\begin{proof}
Let~$\phi$ and a set of subintervals of a line~$L$ be given that satisfy the 
properties in the definition of a fuzzy linear interval strip.

Define~$H$ to be the graph on the vertex set~$\phi(V(G))$ in which two 
vertices are adjacent, if they are consecutive on the line~$L$. 
Note that~$H$ is a path.

Let~$\isg$ be the set of subtrees of~$H$ that are induced by the intersection 
of~$V(H)$ with one of the given intervals. We claim that the 
triple~$(H,\phi,\isg)$ is a reconstruction of~$G$. 
Properties~\ref{prop:if:adj:subgraph}--\ref{prop:only:one:H:if:deg:greater:2} 
follow as in the proof of Lemma~\ref{lem:fuzzy:circ:has:resonstr}.
The fact that the ends do not have the same image under~$\phi$ as any other 
vertex of~$G$ follows from the corresponding additional requirement in the 
definition of fuzzy linear interval strip.

It remains to fulfill the requirement that ends are not part of fuzzy pairs. Let~$a\in V(G)$ be an end of the strip. 
We alter the reconstruction so that end~$a$ is not part of a fuzzy pair. Repeating the alteration with the other end then yields a reconstruction with the required property. If~$a$ is not part of a fuzzy pair, no alteration is required. 
Thus, suppose~$\{a,v\}$ is a fuzzy pair and~$T\in \isg$ is a tree such that~$\phi(a),\phi(v')\in V(T)$. Note that~$T$ is a path with distinct endpoints~$\phi(a)$ and~$\phi(v)$. Since in the fuzzy linear interval representation there is only one interval that has~$\phi(a)$ as an endpoint, there is only one such tree~$T$. Let~$h\in V(T)$ 
be the neighbor of~$\phi(v)$ in~$T$. We subdivide the 
edge~$\{\phi(v),h\}$, adding a new vertex~$h'$, in~$H$ as well as in every 
tree~$T'\in \isg$ that contains the 
edge~$\{\phi(v),h\}$. We then alter~$\phi$ by setting~$\phi(w) = h'$ for all 
vertices in~$w\in \phi^{-1}(\phi(v)) \cap N(a)$. We replace~$T$ by two 
new trees~$T_1 := T - \phi(v)$ and~$T_2 := T - \phi(a)$.
If a tree~$T'\in \isg \setminus \{T\}$ originally contained~$\phi(v)$ but 
not~$h$, we add~$h'$ and the edge~$\{\phi(v),h'\}$ to~$T'$.
This gives a new reconstruction of~$G$ in which~$a$ is not part of a fuzzy pair.
Note that the property that ends do not 
have the same image under~$\phi$ as any other vertex of~$G$ is maintained.
\end{proof}

\begin{definition}[Chudnovsky, Seymour \cite{DBLP:conf/bcc/ChudnovskyS05}]
The \emph{composition of two strips} is defined as follows:
Suppose that~$(G, a, b)$ and~$(G', a', b')$ are two strips. We compose them as 
follows.
Let~$A,B$ be the set of vertices of~$G - \{a,b\}$ adjacent in G to~$a, b$ 
respectively, and
define~$A',B'$ similarly. 
Take the disjoint union of~$G - \{a,b\}$ and~$G'- 
\{a',b'\}$ and let~$H$ be the graph obtained from this by adding all possible 
edges between~$A$ and~$A'$ and between~$B$ and~$B'$. Then~$H$ is the 
composition of the two strips.

In general the \emph{composition of strips} is defined as follows:
Start with a graph~$G_0$ with an even number of vertices and which is the disjoint union of complete
graphs, and pair the vertices of~$G_0$. Let the pairs be~$(a_1,b_1),\ldots, (a_n,b_n)$, say. 
For~$i= 1,\ldots,n$, let~$(G_i',a_i',b_i')$ be a strip.
For~$i= 1,\ldots,n$, let~$G_i$ be the graph obtained by 
composing~$(G_{i-1}, a_i, b_i)$ and~$(G'_i, a'_i, b'_i)$; 
then the resulting graph~$G_n$ is called  a composition of the 
strips~$(G_i',a_i',b_i')$~$(1 \leq i \leq n)$.

\end{definition}

The structure theorem of Chudnovsky and 
Seymour says that every connected quasi-line graph is a fuzzy circular 
interval graph, or a composition of fuzzy linear interval strips.

\begin{lemma}\label{lem:fuzzy:lin:has:reconstr}
A graph that is the composition of fuzzy linear interval strips has a 
reconstruction.
\end{lemma}

\begin{proof}
Let~$G$ be the graph obtained by composing the fuzzy linear interval 
strips~$(G'_i, a'_i, b'_i)$ with~$i \in\{1,\ldots,n\}$.
By Lemma~\ref{lem:reconst:end:not:in:fuz} each strip~$(G'_i, a'_i, b'_i)$ has 
a reconstruction~$(H_i,\phi_i,\isg_i)$ for which~$H_i$ is a path and for which 
the ends do not have the same image under~$\phi_i$ as any other vertex 
of~$H_i$. Additionally we can require that the ends are not part of fuzzy vertex pairs. Without loss of generality all graphs~$H_i$ are disjoint.
Define a map~$\psi$ from the set of 
ends~$\{a'_1,\ldots,a'_n,b'_1,\ldots,b'_n\}$ to~$V(G_0)$ by mapping~$a'_i$ 
to~$a_i$ and~$b'_i$ to~$b_i$.
Let~$H$ be the graph obtained by forming the disjoint union of all 
graphs~$H_i$ and then adding all edges~$\{u,v\}$, where~$u$ and~$v$ are distinct elements of~$\{a'_1,\ldots,a'_n,b'_1,\ldots,b'_n\}$ with~$\{\psi(u),\psi(v)\}\in E(G_0)$. We let
\[\isg := \bigcup_{i= 1} ^n \{T \in \isg_i \mid V(T)\text{ contains no end of } 
H_i\} \cup \{ T_C \mid C \text{ a connected component of } G_0\}\] where for a 
component~$C$ the graph~$T_C$ is defined as some subtree of~$H$ spanning all vertices
for which there exists a tree~$T$ in some~$\isg_i$ whose vertex set~$V(T)$ contains a vertex that maps to~$C$ via~$\psi$.

Let~$\phi\colon V(G) \rightarrow V(H)$ be the map given by~$\phi(v) = 
\phi_i(v)$ if~$v\in V(G_i)$.

We claim that~$(H,\phi,\isg)$ is a reconstruction of~$G$.
We first argue that 
Property~\ref{prop:if:adj:subgraph} holds. For two adjacent vertices~$v,v'$ 
for which~$\phi(v)$ and~$\phi(v')$ lie in a common graph~$H_i$, there is a 
tree~$T\in \isg_i$ that contains~$\phi(v)$ and~$\phi(v')$. This tree is also 
contained in~$\isg$.
Suppose now that the vertices~$v$ and~$v'$ are adjacent but~$\phi(v)$ and~$\phi(v')$ lie in 
different graphs~$H_i$ and~$H_j$.
This implies that there are ends~$w$ and~$w'$ mapping to the same 
component~$C$ of~$G_0$ such that~$w$ is adjacent to~$v$ and~$w'$ is adjacent 
to~$v'$. Thus~$T_C$ contains~$\phi(v)$ and~$\phi(v')$. We argue that the 
remaining properties hold: Properties~\ref{prop:fuzzy:means:pathends} 
and~\ref{prop:paired:fuzzyness} follow from the fact that for any 
component~$C$ of~$G_0$ all preimages of vertices in~$T_C$ are adjacent and 
from the fact that the properties hold for the reconstructions of the strips.
Property~\ref{prop:only:one:H:if:deg:greater:2} follows since the graphs~$H_i$ 
in the reconstructions of the strips are paths.
\end{proof}

\begin{lemma}\label{lem:char:one:direction}
Every quasi-line graph has a reconstruction.
\end{lemma}
\begin{proof}
Since we can obtain a reconstruction for a disconnected graph from
reconstructions of its components, it suffices to show the lemma for connected graphs. 
Thus, suppose~$G$ is a connected quasi-line graph. 
By the structure 
theorem~\cite{DBLP:conf/bcc/ChudnovskyS05} the graph~$G$ is a fuzzy circular interval 
graph, in which case~$G$ has a reconstruction by 
Lemma~\ref{lem:fuzzy:circ:has:resonstr}, or a composition of fuzzy linear 
interval strips, in which case~$G$ has reconstruction by 
Lemma~\ref{lem:fuzzy:lin:has:reconstr}.
\end{proof}

To prove the converse of Theorem~\ref{thm:quasi:line:char} we first show 
that every graph that has a reconstruction also has a reconstruction that 
satisfies additional properties.
We say a reconstruction is \emph{pleasant} if the following five additional 
properties hold:

\begin{enumerate}
\setcounter{enumi}{4}
\item If~$\{v,v'\} \subseteq V(G)$ is a fuzzy vertex pair, then there is 
a vertex~$v'' \in V(G)$ adjacent to~$v$ such that~$\phi(v'')= \phi(v')$. 
\label{prop:fuzzy}
\item For every vertex~$v$ of~$G$ the vertex~$\phi(v)$ has degree at most~$2$
in~$H$. \label{prop:degree}
\item Each tree~$T\in \isg$ has at least two vertices and for every leaf~$h$
of~$T$ there is a vertex~$v\in V(G)$ with~$\phi(v) =
h$.\label{prop:trees:with:useful:leaves}
\item $E(H) = \bigcup_{T\in \isg} E(T)$. \label{prop:all:edges}
\item There are no two distinct trees~$T_1,T_2 \in \isg$ such that~$V(T_1)\subseteq V(T_2)$.\label{prop:no:inclusion}
\end{enumerate}

Figure~\ref{figure1} shows a reconstruction~$(H,\phi,\isg)$ which, due to violations of Properties~\ref{prop:fuzzy} and~\ref{prop:degree}, is not pleasant. It also shows a pleasant reconstruction~$(H',\phi',\isg')$ derived from this. In fact, every reconstruction of a graph~$G$ can be altered into a pleasant reconstruction of~$G$.

\begin{lemma}\label{lem:has:pleasant:reconst}
A graph that has a reconstruction also has a pleasant reconstruction.
\end{lemma}

\begin{proof}

(Property~\ref{prop:fuzzy}.) Suppose~$\{v,v'\}\subseteq V(G)$ is a fuzzy 
vertex pair but there is no 
vertex~$v''\in V(G)$ adjacent to~$v$ for which~$\phi(v'')= \phi(v')$. We now 
construct a different reconstruction of~$G$ that has at least one fuzzy 
vertex pair less than the original reconstruction. For this let~$T\in \isg$ be 
such 
that~$\phi(v),\phi(v')\in V(T)$. The properties of a reconstruction imply 
that~$T$ is a path with endpoints~$\phi(v)$ and~$\phi(v')$. By 
Property~$\ref{prop:three:prime}'$
,~$T$ is the only tree in~$\isg$ that 
contains both~$\phi(v)$ and~$\phi(v')$. Let~$h\in V(T)$ 
be the neighbor of~$\phi(v)$ in~$T$. We subdivide the 
edge~$\{\phi(v),h\}$, adding a new vertex~$h'$, in~$H$ as well as in every 
tree~$T'\in \isg$ that contains the 
edge~$\{\phi(v),h\}$. We then alter~$\phi$ by setting~$\phi(w) = h'$ for all 
vertices in~$w\in \phi^{-1}(\phi(v))\setminus \{v\}$. We replace~$T$ by two 
new trees~$T_1 := T - \phi(v)$ and~$T_2 := T - \phi(v')$.
If a tree~$T'\in \isg \setminus \{T\}$ originally contained~$\phi(v)$ but 
not~$h$ we add~$h'$ and the edge~$\{\phi(v),h'\}$ to~$T'$.
This gives a new reconstruction of~$G$ that has one fuzzy pair less. 
By repeating the modification we eventually obtain a reconstruction that
satisfies Property~\ref{prop:fuzzy}. 

(Property~\ref{prop:degree}.) We now argue that we can avoid that~$\phi$ maps 
vertices of~$G$ to vertices 
of~$H$ of degree larger than~$2$.
Suppose~$v$ is mapped to a vertex of degree larger than~$2$ in~$H$. By 
definition, there is at most one tree~$T\in \isg$ that contains~$\phi(v)$.
We add a new auxiliary vertex~$h$ and the edge~$\{h,\phi(v)\}$ to~$H$ and 
to~$T$. For every vertex~$w \in \phi^{-1}(\phi(v))$, we then 
redefine~$\phi$ by setting~$\phi(w) := h$.
By repeating this operation we obtain 
a reconstruction that does not map any vertices of~$G$ to vertices of~$H$ of 
degree larger than~$3$.
Property~\ref{prop:fuzzy} is maintained by the modification.

(Property~\ref{prop:trees:with:useful:leaves}). We first argue that we can 
find a reconstruction in which for every~$T\in \isg$ there 
are vertices~$v$ and~$v'$ such that~$\phi(v),\phi(v')\in 
V(T)$ and~$\phi(v)\neq\phi(v')$. Let~$T\in \isg$ be a tree that does not have 
this property. If there is at most one vertex~$v$ with~$\phi(v)\in V(T)$, then 
we simply delete~$T$ from~$\isg$. Suppose otherwise. There is a vertex~$h$ 
of~$T$ and two vertices~$v,v'$ of~$G$, such that~$\phi(v) = \phi(v') = h$. If 
there is another tree~$T'\in \isg$ with~$h\in V(T')$ we can delete~$T$ and 
still have a valid reconstruction. Otherwise we add a new auxiliary 
vertex~$h'$ to~$H$ and add the edge~$\{h',h\}$.
We then redefine~$T$ to be the tree induced by~$h'$ and~$h$ and 
redefine~$\phi(v)$ to be~$h'$. This gives us a valid reconstruction and the 
new~$T$ has two vertices which are images of vertices of~$G$ under~$\phi$.

To achieve that all leaves of the trees are images of vertices of~$G$, from
every tree~$T\in \isg$ 
we repeatedly delete leaf vertices~$h\in V(T)$ for which no vertex~$v\in V(G)$
exists with~$\phi(v) = h$. This yields a reconstruction that satisfies
Property~\ref{prop:trees:with:useful:leaves}. 
Note that these modifications of the reconstruction maintain
Properties~\ref{prop:fuzzy} and~\ref{prop:degree}.

(Properties~\ref{prop:all:edges} and~\ref{prop:no:inclusion}.) By deleting all the edges of~$H$ that are not in~$\bigcup_{T\in \isg} E(T)$ and repeatedly removing trees~$T\in \isg$ whose vertex set is contained in some other tree~$T'\in \isg $ we obtain a reconstruction that satisfies Properties~\ref{prop:all:edges} and \ref{prop:no:inclusion} while maintaining Properties~\ref{prop:fuzzy}--\ref{prop:trees:with:useful:leaves}.
\end{proof}

Let~$(H,\phi,\isg)$ be a reconstruction. For every vertex~$h\in V(H)$ and 
every edge~$e \in E(H)$ incident with~$h$, let~${\mathcal{P}}_{h,e}$ be the 
set of paths~$P$ in~$H$ that start at~$h$, continue with~$e$ as first edge, and for which there 
exists a tree~$T\in \isg$ that contains~$P$ as a subgraph.

\begin{lemma}\label{lemma:unique:Hew}
Let~$(H,\phi,\isg)$ be a pleasant reconstruction. Let~$h\in V(H)$ be a vertex 
and~$e \in E(H)$ an edge incident with~$h$.

\begin{enumerate}
\item There is a unique tree~$T_{h,e} \in \isg$
that contains all paths in~${\mathcal{P}}_{h,e}$ as subgraphs.
\item If~$h'$ is a leaf of~$T_{h,e}$ and~$e'$ is the edge of~$T_{h,e}$ 
incident with~$h'$, then~$T_{h,e} = T_{h',e'}$.
\end{enumerate}
\end{lemma}

\begin{proof}
By Property~\ref{prop:all:edges} of pleasant reconstructions, at least one 
tree in~$\isg$ must contain~$e$. 
It suffices to show the existence of a unique tree that contains all maximal 
paths in~${\mathcal{P}}_{h,e}$. 
We first assume that there are at least two distinct maximal paths 
in~${\mathcal{P}}_{h,e}$. Note that any two distinct maximal paths 
in~${\mathcal{P}}_{h,e}$ share a vertex that has degree at least 3 in~$H$. 
Thus, by Property~\ref{prop:only:one:H:if:deg:greater:2}, each path in~${\mathcal{P}}_{h,e}$ is contained in exactly one tree in~$\isg$ and this tree is the same tree for all paths in~${\mathcal{P}}_{h,e}$.

Now suppose there is only one maximal path in~${\mathcal{P}}_{h,e}$. Then, by 
definition of~${\mathcal{P}}_{h,e}$, there is a tree that contains this 
maximal path and therefore contains all paths in~${\mathcal{P}}_{h,e}$. 
If~$T,T'\in \isg$ were two distinct trees with this property, then due to Property~\ref{prop:no:inclusion} the trees~$T$ 
and~$T'$ would share a vertex of degree~$3$. However, this is not the case, by 
Property~\ref{prop:only:one:H:if:deg:greater:2}.
This shows uniqueness and thus the first part of the lemma.

To prove the second part of the lemma, it suffices to show that all paths 
in~${\mathcal{P}}_{h,e}$ are contained in~$T_{h',e'}$.
Suppose~$P \in {\mathcal{P}}_{h,e}$. Since~$T_{h,e}$ is a tree, there are two 
(possibly equal) paths~$P_1$ and~$P_2$ in~$T_{h,e}$ starting at vertex~$h'$ 
with first edge~$e'$ that together cover~$P$. By definition both paths are 
contained in~$T_{h',e'}$ and thus~$P$ is a subgraph of~$T_{h',e'}$.
\end{proof}

\begin{proof}[Proof of Theorem~\ref{thm:quasi:line:char}]
One direction of the theorem has already been established by 
Lemma~\ref{lem:char:one:direction}. 

For the converse let~$G$ be a graph that has a reconstruction.
By Lemma~\ref{lem:has:pleasant:reconst} the graph~$G$ has a pleasant 
reconstruction~$(H,\phi,\isg)$. 
Let~$v$ be a vertex of~$G$. We show that all edges incident with~$v$ can be covered 
by two cliques.

Since the reconstruction is pleasant,~$\phi(v)$ has degree at most~$2$.
If~$\phi(v)$ has degree~$0$, then~$v$ is an isolated vertex or the 
set~$\phi^{-1}(\phi(v))$ is a clique that contains all neighbors of~$v$.
Suppose now~$\phi(v)$ has degree at least 1. If~$v$ is contained in a fuzzy 
pair, let~$T \in \isg$ be a path such that there is a vertex~$u$ non-adjacent 
to~$v$ for which~$\phi(v)$ and~$\phi(w)$ form the endpoints of~$T$.
Otherwise let~$e$ be an edge incident with~$v$ and let~$T$ be the union of all 
paths in~${\mathcal{P}}_{h,e}$.
We define for~$v$ two cliques~$S_{v}^{1}$ and~$S_{v}^{2}$ in the following way:

We define~$S_{v}^{2} := \{v\} \cup \big(\phi^{-1}(V(T)\setminus 
\{\phi(v)\})\cap N(v)\big)$, and we define~$S_{v}^{1} := \{v\} \cup (N(v) 
\setminus S_{v}^{2})$. By definition, the sets~$S_{v}^{1}$ and~$S_{v}^{2}$ 
together cover all edges incident with~$v$ and it thus suffices to show that they 
are cliques.

The set~$S_{w}^{2}$ forms a clique by Property~\ref{prop:fuzzy:means:pathends} 
of reconstructions. 
To see that the set~$S_{w}^{1}$ is also a clique, it suffices to show that 
every two distinct vertices~$w_1,w_2 \in S_{v}^{1} \setminus \{v\}$ are 
adjacent. If~$\phi(w_1) = \phi(v)$ and~$w_2$ were not adjacent to~$w_1$, 
then~$\{w_1,w_2\}$ would be a fuzzy vertex pair. Then 
Property~$\ref{prop:three:prime}'$ 
would imply~$\phi(w_2) = \phi(u) \subseteq V(T)\setminus \{\phi(w)\}$, which 
implies~$w_2 \notin S_{v}^{1}$, giving a contradiction. Symmetrically, it 
follows from~$\phi(w_2) = \phi(v)$ that~$w_1$ and~$w_2$ are adjacent. Finally 
suppose~$\phi(w_1)$ and~$\phi(w_2)$ are different from~$\phi(w)$. Then there 
are trees~$T_1 , T_2\in \isg\setminus \{T\}$ such that for~$i\in\{1,2\}$, we 
have~$\{\phi(v),\phi(w_i)\} \subseteq V(T_i)$. Since~$\{u,v\}$ is a fuzzy 
pair, for~$i\in\{1,2\}$ we also have~$V(T) \nsubseteq V(T_i)$.
Since~$V(T_1)$ and~$V(T_2)$ cannot share a vertex of degree at least~$3$ 
in~$H$ (Property~\ref{prop:only:one:H:if:deg:greater:2}), 
either~$\phi(w_1),\phi(w_2) \in V(T_1)$ or~$\phi(w_1),\phi(w_2) \in V(T_2)$. 
Without loss of generality say~$\phi(w_1),\phi(w_2) \in V(T_1)$. The pair~$\{w_1,w_2\}$ cannot 
be fuzzy, since the graph~$T_1$ is either not a path or has an endpoint 
in~$V(T)$. Therefore,~$w_1$ and~$w_2$ are adjacent. This shows 
that~$S_{w}^{1}$ is a clique.
\end{proof}

\begin{corollary}
A graph has a pleasant fuzzy reconstruction if and only if it is a quasi-line graph.
\end{corollary}

Concerning reconstructions, we remark that as yet another version of 
Property~\ref{prop:paired:fuzzyness} we could require that no two distinct 
paths in~$\isg$ share an endpoint. This alternative is analogous to the fact that in fuzzy circular interval graphs and 
fuzzy 
linear interval strips the intervals may not share endpoints. However, the 
alternative is not compatible with the properties of pleasant reconstructions. 
Along these lines, there are various properties one might want to add to a 
pleasant reconstruction. For example one can require that the trees in~$\isg$ 
contain at most one vertex of degree~$3$. For the definition of pleasant 
reconstruction we have chosen properties that streamline and simplify our 
proof of Conjecture~\ref{conj:opsut}, which follows next.
\section{Competition numbers and quasi-line graphs} \label{sec:quasi-line}
In this section we show that quasi-line graphs have competition number at 
most~$2$.

We use the following characterization of graphs of competition number at 
most~$2$. The original general characterization for arbitrary competition 
numbers, given by Lundgren and Maybee, contained a slight error which was 
fixed by Kim (see~\cite{KimVadis}). However, for graphs of competition number at 
most~$2$ the original characterization is correct.

\begin{theorem}[Lundgren, Maybee~\cite{MR712931}]\label{thm:char:lund:may}
A graph on~$n$ vertices
has competition number at most~$2$ if and only if there exists a 
clique edge covering~$C_1,\ldots,C_n$ and an ordering of the 
vertices~$v_1,\ldots,v_n$ such that~$v_i \notin C_j$ for~$i\geq j+2$.
\end{theorem}

We call the set of edges incident with a vertex~$v$ the \emph{vertex star of~$v$}.
For a graph~$G$, we say a sequence of cliques~$C_1,\ldots,C_t$ 
\emph{incrementally 
covers} a subset of vertices~$V'\subseteq V(G)$ if~$|V'| \geq t$ and the 
following holds: The cliques 
cover all edges incident with~$V'$ and for all~$i\in \{1,\ldots,t-1\}$ the 
cliques~$C_1,\ldots,C_{i+1}$ cover vertex stars of at least~$i$ distinct 
vertices 
in~$V'$. 

\begin{theorem}\label{thm:clique:sequence}
A graph~$G$ has competition number at most~$2$ if and only if there exists a 
sequence of cliques that incrementally covers~$V(G)$.
\end{theorem}
\begin{proof}
Suppose~$G$ is a graph on~$n$ vertices. We show that the theorem is equivalent 
to Theorem~\ref{thm:char:lund:may}. That is, we show that there exists a 
clique edge covering~$C_1,\ldots,C_n$ and an ordering of the 
vertices~$v_1,\ldots,v_n$ such that~$v_i \notin C_j$ for~$i\geq j+2$ if and 
only if there exists a 
sequence of cliques that incrementally covers~$V(G)$.

For the one direction suppose~$C_1,\ldots,C_n$ is a clique edge 
covering and~$v_1,\ldots,v_n$ an ordering of the vertices such that~$v_i 
\notin C_j$ for~$i \geq j+2$. Then~$C_n,\ldots,C_1$ incrementally cover~$V(G)$.

For the converse, suppose~$C_1,\ldots, C_{n}$ is a sequence of cliques that 
incrementally covers~$V(G)$ and~$v_1,\ldots,v_n$ is an ordering of the
vertices of~$G$ such that for all~$i\in \{1,\ldots,n-1\}$ the vertex star 
of~$v_i$ is 
covered by~$C_1,\ldots, C_{i+1}$. For~$i \in \{1,\ldots,n-2\}$ let~$C'_i := 
C_{n+1-i}\setminus \{v_1,\ldots,v_{n-i-1}\}$ and let~$C'_{n-1} := C_2$ 
and~$C'_n := C_1$. Further, for~$i\in \{1,\ldots,n\}$ let~$v'_i := v_{n-i+1}$. 
The sequence of cliques~$C'_1,\ldots, C'_{n}$ is a clique edge covering, since 
for each~$C_i$ we only remove vertices whose vertex stars are already covered 
by the cliques~$C_1,\ldots,C_{i-1}$. Moreover the sequence of 
cliques~$C'_1,\ldots, C'_{n}$
has the property that~$v'_i \notin C'_j$ for~$i\geq j+2$. 
\end{proof}

\begin{lemma}\label{thm:characterization:inc:covering}
If~$G$ is a vertex minimal quasi-line graph with~$\cp(G) >2$,
then there is no sequence 
of cliques that incrementally covers a subset of the vertices of~$G$.
\end{lemma}
\begin{proof}
Suppose~$G$ is a quasi-line graph with competition number larger than~$2$ and 
there is a sequence of cliques~$C_1,\ldots,C_t$ that incrementally covers a 
subset~$V'\subseteq V(G)$. If~$|V'|>t$, then there is a strict 
subset~$V''\subsetneq V'$ which is also incrementally covered 
by~$C_1,\ldots,C_t$. Thus, w.l.o.g., we may assume~$|V'| = t$. This implies 
that there is an ordering~$v_1,\ldots,v_t$ of the set~$V'$ such that 
for~$i \in \{1,\ldots,t-1\}$ the vertex star of~$v_i$ is covered 
by~$C_1,\ldots,C_{i+1}$ and such that the star of~$v_t$ is covered 
by~$C_1,\ldots, C_t$. Consider the graph~$G' := G -\{v_1,\ldots,v_t\}$. 
Since~$G$ has minimum size among the quasi-line graphs that have competition 
number larger than~$2$ and since 
induced subgraphs of quasi-line graphs are quasi-line graphs, the graph~$G'$ 
has competition number at most~$2$. 
By Theorem~\ref{thm:clique:sequence} there is an incremental clique 
covering~$C'_1,\ldots,C'_{n-t}$ of~$V(G')$.
The sequence of cliques~$C_1,\ldots,C_t, C'_1,\ldots, C'_{n-t}$ forms an 
incremental clique covering of~$V(G)$, and therefore the competition number 
of~$G$ is at most~$2$, which gives a contradiction. 
\end{proof}

In the rest of this section we show that every quasi-line graph has an incremental clique covering of a 
subset of the vertices. By the previous lemma this shows that there are no minimal counterexamples to Conjecture~\ref{conj:opsut} and thus no counterexamples at all.

\begin{lemma}\label{lem:fuzzyness:removal}
If a graph has a pleasant reconstruction that has a fuzzy vertex pair, then there exists an incremental clique covering of a 
subset of the vertices.
\end{lemma}

\begin{proof}
Suppose~$(H, \phi, \isg)$ is a pleasant reconstruction of a graph~$G$. 
Further suppose~$\{v_1,u_1\}$ is a fuzzy vertex pair and~$T\in \isg$ contains 
both~$v_1$ and~$u_1$. We define for every vertex~$w$ with~$\phi(w) \in 
\{\phi(v_1),\phi(u_1)\}$ two cliques~$S_{w}^{1}$ and~$S_{w}^{2}$ in the 
following way:
We define~$S_{w}^{2} := \{w\} \cup \big(\phi^{-1}(V(T)\setminus 
\{\phi(w)\})\cap N(w)\big)$, and we define~$S_{w}^{1} := \{w\} \cup (N(w) 
\setminus S_{w}^{2})$.
The set~$S_{w}^{2}$ forms a clique by Property~\ref{prop:fuzzy:means:pathends} 
of reconstructions. 
To see that the set~$S_{w}^{1}$ is also a clique, it suffices to show that 
every two distinct vertices~$w_1,w_2 \in S_{w}^{1} \setminus \{w\}$ are 
adjacent. If~$\phi(w_1) = \phi(w)$ and~$w_2$ were not adjacent to~$w_1$, 
then~$\{w_1,w_2\}$ would be a fuzzy vertex pair. Then 
Property~$\ref{prop:three:prime}'$ 
would imply~$\phi(w_2) = \phi(u_1) \subseteq V(T)\setminus \{\phi(w)\}$, which 
implies~$w_2 \notin S_{w}^{1}$, giving a contradiction. Symmetrically, it 
follows from~$\phi(w_2) = \phi(w)$ that~$w_1$ and~$w_2$ are adjacent. Finally 
suppose~$\phi(w_1)$ and~$\phi(w_2)$ are different from~$\phi(w)$. Then there 
are trees~$T_1 , T_2\in \isg\setminus \{T\}$ such that for~$i\in\{1,2\}$, we 
have~$\{\phi(w),\phi(w_i)\} \subseteq V(T_i)$. Since~$\{u_1,v_1\}$ is a fuzzy 
pair, for~$i\in\{1,2\}$ we also have~$V(T) \nsubseteq V(T_i)$.
Since~$V(T_1)$ and~$V(T_2)$ cannot share a vertex of degree at least~$3$ 
in~$H$ (Property~\ref{prop:only:one:H:if:deg:greater:2}), 
either~$\phi(w_1),\phi(w_2) \in V(T_1)$ or~$\phi(w_2),\phi(w_2) \in V(T_2)$. 
W.l.o.g. say~$\phi(w_2),\phi(w_2) \in V(T_1)$. The pair~$\{w_1,w_2\}$ cannot 
be fuzzy, since the graph~$T_1$ is either not a path or has an endpoint 
in~$V(T)$. Therefore,~$w_1$ and~$w_2$ are adjacent. This shows 
that~$S_{w}^{1}$ is a clique.

We claim that~$S_{w}^{1} = S_{w'}^{1}$ if~$\phi(w) = \phi(w')$: Indeed, for 
every vertex in~$v\in S_{w}^{1}$ either~$\phi(v) = \phi(w)$, in which 
case~$v\in S_{w'}^{1}$, or there is a tree~$T'$ different from~$T$ 
with~$\{\phi(w),\phi(v)\}\subseteq V(T')$. 
Since~$\{u_1,v_2\}$ is a fuzzy pair, by Property~$\ref{prop:three:prime}'$ the 
pair~$\{v,w'\}$ is not fuzzy. Therefore~$v$ and~$w'$ are adjacent, and 
thus~$v\in S_{w'}^{1}$.

We also claim that for every~$u\in \phi^{-1}(\phi(u_1))$ and~$x \in N(u)$, 
either~$\{u,x\} \subseteq S_{u}^{1} = S_{u_1}^{1}$, or there is a~$v\in 
\phi^{-1}(\phi(v_1))$ such that~$\{u,x\} \subseteq S_{v}^{2}$: Indeed, if~$x 
\notin S_{u}^{1}$, then~$\phi(x)\in V(T)\setminus \{\phi(u)\}$. 
If~$\phi(x) = \phi(v_1)$ then, we can choose~$v$ as~$x$.
Otherwise let~$v$ be a vertex adjacent to~$u$ with~$\phi(v) = \phi(v_1)$. This 
vertex exists by Property~\ref{prop:fuzzy} of a pleasant reconstruction 
and~$\{u,x\} \subseteq S_{v}^{2}$ holds by definition of~$S_{v}^{2}$.

Since the reconstruction is pleasant, by Property~\ref{prop:fuzzy}, the 
preimages~$\phi^{-1}(\phi(v_1))$ and~$\phi^{-1}(\phi(u_1))$ have size at 
least~$2$. Suppose~$\phi^{-1}(\phi(v_1)) = \{v_1,\ldots,v_t\}$. 
By our previous two claims, the clique sequence~$S_{v_1}^{1}, 
S_{v_1}^{2},S_{v_2}^{2},\ldots, 
S_{v_t}^{2}, S_{u_1}^{1}$ incrementally covers the set~$\phi^{-1}(\phi(v_1)) 
\cup \phi^{-1}(\phi(u_1))$. Thus, by 
Lemma~\ref{thm:characterization:inc:covering}, the graph~$G$ is not vertex 
minimal with~$\cp(G)>2$.
\end{proof}

\begin{lemma}\label{lem:nicely:behaved:cliques}
Let~$G$ be a quasi-line graph without simplicial vertices and~$(H,\phi,\isg)$ 
a pleasant 
reconstruction of~$G$ without fuzzy pairs.
It is possible to associate with every vertex~$v \in V(G)$ a pair 
of cliques~$C_v^1, C_v^2$, such that the following holds:
\begin{itemize}
\item For every vertex~$v \in V(G)$ the cliques~$C_v^1$ and~$C_v^2$ cover all 
edges incident with~$v$.
\item For every~$v\in V(G)$
 and every~$j\in \{1,2\}$ there 
is~$v'\neq v$ and~$j' \in \{1,2\}$
such that~$C_{v'}^{j'} = {C^{j}_{v}}$.
\end{itemize}
\end{lemma}

\begin{proof}
Note that, since the reconstruction does not have fuzzy vertex pairs, any set 
that is of the form~$\phi^{-1}(V(T))$, with~$T\in \isg$, forms a clique in~$G$.

We now argue that for every~$v\in V(G)$ the degree of~$\phi(v)$ in~$H$ is 
exactly 2. If this is not the case, then, by Property~\ref{prop:degree} of a 
pleasant reconstruction, the degree of~$\phi(v)$ is smaller than 2. 
If~$\phi(v)$ has degree~$1$ and is incident with edge~$e'$, 
then~$\phi^{-1}(V(T_{\phi(v),e'}))$ is a clique that covers the vertex star 
of~$v$ and thus~$v$ is simplicial. If~$\phi(v)$ has degree~$0$, then, by 
Property~\ref{prop:trees:with:useful:leaves},~$v$ is of degree 0 and thus 
simplicial.

We associate two cliques~$C_{v}^{1}$,~$C_{v}^{2}$ with every vertex~$v\in 
V(G)$. 
Since~$\phi(v)$ has degree exactly~$2$, there are two edges~$e',e''$ incident 
with~$\phi(v)$. 
We define~$C_{v}^{1} := \phi^{-1}(V(T_{\phi(v),e'}))$ and~$C_{v}^{2} := 
\phi^{-1}(V(T_{\phi(v),e''}))$.

To show the first claim of the lemma suppose~$v$ and~$v'$ are adjacent in~$G$. 
We show that 
one of the cliques~$C_v^1$ or~$C_v^2$ contains~$v'$. 
If~$\phi(v) = \phi(v')$, then~$v'\in C_v^1$. Otherwise, by definition of 
the reconstruction, there is a tree~$T\in \isg$ that contains both~$v$ 
and~$v'$. 
Thus there is a path from~$\phi(v')$ to~$\phi(v)$ that lies entirely in~$T$.
Let~$e'$ be the edge of this path incident with~$\phi(v)$, then the path is 
contained in~$T_{\phi(v),e'}$. 
Thus one of the cliques~$C_v^1$ or~$C_v^2$ contains~$v'$.

To show the second claim suppose~$v$ is a vertex of~$G$ and suppose~$C\in 
\{C_v^1, C_v^2\}$. There is an edge~$e$ incident with~$\phi(v)$ such that~$C = 
\phi^{-1}(V(T_{\phi(v),e}))$. Let~$v'\in V(G)$ be a vertex for 
which~$\phi(v')$ is a leaf of~$V(T_{\phi(v),e})$ 
such that~$\phi(v')\neq \phi(v)$. 
Such a vertex exists by Property~\ref{prop:trees:with:useful:leaves}.
Let~$e'$ be the edge of~$T_{\phi(v),e_1}$ incident with~$\phi(v')$. By 
Lemma~\ref{lemma:unique:Hew},~$T_{\phi(v),e}= T_{\phi(v'),e'}$ and we 
conclude~$C = \phi^{-1}(V(T_{\phi(v'),e'})) \in \{ C^{1}_{v'} ,C^{2}_{v'}\}$.
\end{proof}

\begin{lemma}\label{lem:has:inc:covering}
For every quasi-line graph~$G$ there is an incremental clique covering of a 
subset of the vertices.

\end{lemma}
\begin{proof}
Since for any simplicial vertex~$v$ the clique~$N(v)\cup \{v\}$ incrementally 
covers~$\{v\}$, we can assume that~$G$ does not have
simplicial vertices. Furthermore, by 
Lemmas~\ref{thm:characterization:inc:covering} 
and~\ref{lem:fuzzyness:removal}, it suffices to show the statement for 
quasi-line graphs which have a pleasant 
reconstruction~$(H,\phi,\isg)$ without fuzzy pairs.
For this reconstruction we can apply Lemma~\ref{lem:nicely:behaved:cliques}, 
and associate every vertex~$v$ with two cliques~$C^{1}_{v}$ and~$C^{2}_{v}$. 

We consider the following bipartite graph: One bipartition class consists of 
all vertices~$v\in V(G)$. The other bipartition class consists of all 
cliques~$C$ of~$G$ for which there exists a vertex~$v\in V(G)$ and a~$j\in 
\{1,2\}$ such that~$C = C^{j}_{v}$. The edge set is~$\{\{v,C^{i}_v\} \mid v\in V(G), i\in\{1,2\}   \}$.

This bipartite graph has minimum degree~$2$, since every vertex is adjacent to 
two cliques and since by Lemma~\ref{lem:nicely:behaved:cliques} every 
clique is adjacent to at least two vertices.
Let~$C_1,v_1,\ldots,C_t,v_t$ be a cycle in the bipartite graph. By 
construction, for every~$k\in \{1,\ldots,t-1\}$, the cliques~$C_k$ 
and~$C_{k+1}$ cover the vertex star of~$v_k$. Moreover~$C_1$ and~$C_t$ cover 
the vertex star of~$v_t$. Thus the sequence~$C_1,\ldots,C_t$ incrementally 
covers the set~$\{v_1,\ldots,v_t\}$.
\end{proof}

\begin{theorem}
If~$G$ is quasi-line graph, then~$G$ has competition number at most~$2$.
\end{theorem}
\begin{proof}
By Lemma~\ref{lem:has:inc:covering} every quasi-line graph~$G$ has an 
incremental clique covering of a 
subset of the vertices. Lemma~\ref{thm:characterization:inc:covering} then 
shows 
that there are no vertex minimal quasi-line graphs with competition number 
greater than~$2$, and thus all quasi-line graphs have competition number at 
most~$2$.
\end{proof}

\begin{corollary}[Chen, Jacobson, K\'{e}zdy, Lehel, Scheinerman, 
Wang~\cite{Chen200017}]
Every quasi-line graph on~$n$ vertices has a clique edge covering consisting 
of~$n$ cliques.
\end{corollary}

\section{Competition numbers and holes}\label{sec:competition_number_bound}
To study the relationship between the competition number and the number of holes in a graph, we introduce some additional terminology.
Recall that a hole in a graph is a 
subset of at least~$4$ vertices that induces a simple cycle.
We denote by~$\hole(G)$ the number of holes in a graph~$G$ and 
by~$\hole(v)$ the number of holes containing the vertex~$v$.
For a graph~$G$, we let~$\omega(G)$ be the size of 
a largest clique of~$G$. 
For a vertex~$v$ in a graph~$G$ we define its 
\emph{simplicial defect} as~$\defect(v) := |N(v)| - \omega(G[N(v)])$. 
The terminology is justified because a 
vertex is simplicial if and only if its defect is~$0$.

We say a vertex~$v$ is \emph{good} if for every two non-adjacent
vertices~$u,u'\in N(v)$ there exists a hole that contains~$v$,~$u$ and~$u'$.
Note that such a hole cannot contain any other vertices from~$N(v)$ 
besides~$u$ and~$u'$. This implies that~$\defect(v) \leq \hole(v)$ for every 
good vertex~$v$.

Next we show that every graph has a good vertex. The proof is a generalization
of a proof by Dirac~\cite{Dirac} that shows
that every non-complete chordal graph has two non-adjacent simplicial vertices.

\begin{lemma}\label{good:lemma}
If~$G$ is a non-complete graph, then~$G$ contains two non-adjacent 
good vertices. 
\end{lemma}

\begin{proof}
We show the lemma by induction on the number of vertices of~$G$.

Let~$G$ be a smallest non-complete graph which is a counterexample
to the lemma, and let~$S$ be a minimal vertex separator (which
is empty in the case that~$G$ is disconnected). Let~$(A,B)$ be a
non-trivial partition of the vertices in~$G-S$ such that there is no
edge of~$G$ with an endpoint in~$A$ and an endpoint in~$B$. Consider
the graph~$G_A$ which is obtained in the following way: Start
with the subgraph of~$G$ that is induced by the vertex set~$A\cup S$,
then add all edges between vertices of~$S$, as detailed in
Figure~\ref{fig:construction_of_GA}. (Be aware that~$G_A$ is not
necessarily an induced subgraph of~$G$.)

\begin{figure}[htb]
\centering

\begin{picture}(0,0)%
\includegraphics{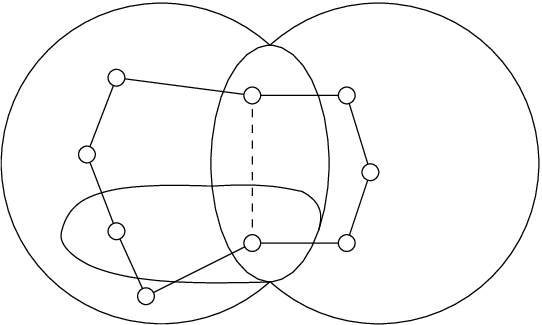}%
\end{picture}%
\setlength{\unitlength}{4144sp}%
\begingroup\makeatletter\ifx\SetFigFont\undefined%
\gdef\SetFigFont#1#2#3#4#5{%
  \reset@font\fontsize{#1}{#2pt}%
  \fontfamily{#3}\fontseries{#4}\fontshape{#5}%
  \selectfont}%
\fi\endgroup%
\begin{picture}(4114,2463)(4964,-2755)
\put(5806,-1411){\makebox(0,0)[lb]{\smash{{\SetFigFont{12}{14.4}{\rmdefault}{\mddefault}{\updefault}{\color[rgb]{0,0,0}$A$}%
}}}}
\put(6166,-2626){\makebox(0,0)[lb]{\smash{{\SetFigFont{12}{14.4}{\rmdefault}{\mddefault}{\updefault}{\color[rgb]{0,0,0}$v$}%
}}}}
\put(7876,-1411){\makebox(0,0)[lb]{\smash{{\SetFigFont{12}{14.4}{\rmdefault}{\mddefault}{\updefault}{\color[rgb]{0,0,0}$B$}%
}}}}
\put(6886,-2356){\makebox(0,0)[lb]{\smash{{\SetFigFont{12}{14.4}{\rmdefault}{\mddefault}{\updefault}{\color[rgb]{0,0,0}$u'$}%
}}}}
\put(6166,-2086){\makebox(0,0)[lb]{\smash{{\SetFigFont{12}{14.4}{\rmdefault}{\mddefault}{\updefault}{\color[rgb]{0,0,0}$N(v)$}%
}}}}
\put(7021,-1546){\makebox(0,0)[lb]{\smash{{\SetFigFont{12}{14.4}{\rmdefault}{\mddefault}{\updefault}{\color[rgb]{0,0,0}$S$}%
}}}}
\put(5626,-2131){\makebox(0,0)[lb]{\smash{{\SetFigFont{12}{14.4}{\rmdefault}{\mddefault}{\updefault}{\color[rgb]{0,0,0}$u$}%
}}}}
\end{picture}%

\caption{The figure illustrates the construction of~$G_A$ used in the proof of 
Lemma~\ref{good:lemma}. Edges within the vertex separator~$S$ are added 
(dashed line) and~$G_A$ is then obtained as the subgraph  induced by~$A\cup 
S$. The figure also illustrates how, in the proof, the induced cycle~$C$
in~$G_A$ is altered to obtain an induced cycle in~$G$ which contains~$v$,~$u$
and~$u'$.
}
\label{fig:construction_of_GA}
\end{figure}
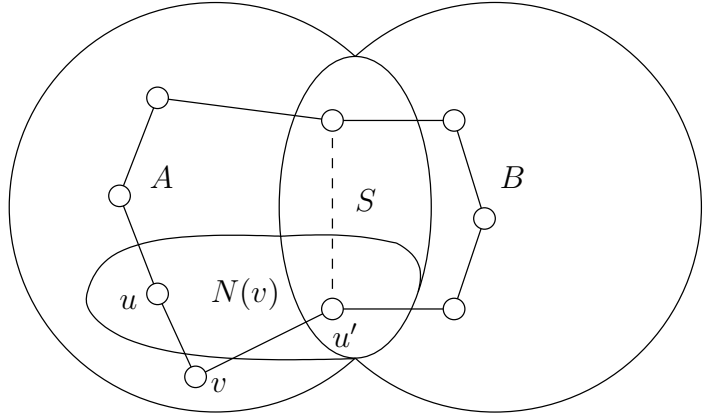

We next argue that~$G_A$ has a good vertex~$v$ in~$A$. 
Obviously, if~$G_A$ is a complete
graph then every vertex of~$G_A$ is good. Otherwise, by the induction hypothesis,~$G_A$ has two
non-adjacent good vertices. Since~$S$ is a clique of~$G_A$, by
construction, at least one of the good vertices is in~$A$.

We now show that~$v$ is also good in~$G$. From the
definition, this is equivalent to showing that every path~$u,v,u'$
of distinct vertices lies on an induced cycle (perhaps a triangle).
Let~$u,u'$ be distinct
neighbors of~$v$ in~$G$. Since~$v$ is good in~$G_A$, there is
an induced cycle~$C$ in~$G_A$ that includes~$\{v,u,u'\}$. Since
the cycle is induced,~$C$ contains at most 2 vertices of~$S$.
This means that~$C$ is also an induced cycle in~$G$ unless~$C\cap S$
is an edge~$\{x,y\}$ that is not an edge of~$G$. In that case, if
$x$ was not adjacent to every component of~$G-S$, then~$S-x$ would be
a smaller separator, and similarly for~$y$, so there must be a path
$P$ in~$G$ from~$x$ to~$y$ whose internal vertices all lie in~$B$.
Choosing such a path of minimal length, we find that~$C\cup V(P)$ is
an induced cycle in~$G$.

Thus~$v$ is a vertex in~$A$ that is good in the graph~$G$. By
symmetry there is also a good vertex in~$B$. Since vertices in~$A$
are not adjacent to vertices in~$B$, this concludes the proof.
\end{proof}

\begin{corollary}\label{all:good:corollary}
Every graph has a good vertex.
\end{corollary}
\begin{proof}
If a graph is complete, then every vertex is good. If a graph is not complete,
it contains a good vertex by Lemma~\ref{good:lemma}.
\end{proof}

\begin{lemma}\label{lem:comp:nr:vs:defect}
If~$v$ is a vertex of a graph~$G$ on at least 2 vertices, then~$\cp(G) \leq 
\max \{1,\cp(G-v)\}   
 + \defect(v)$.   
\end{lemma}
\begin{proof}
Let~$t := \max\{1,\cp(G-v)\}$. By definition, there exists an acyclic directed graph~$D_v$
with competition graph~$(G-v) \dunion t  K_1$, where~$t  K_1$ is an
empty graph on~$t$ vertices. By possibly deleting unncessary arcs, we can achieve that~$D_v$ has a vertex~$u \in V(t  K_1)$ that does not have any outgoing arcs. Moreover, we can require that~$v\notin V(t  K_1)$, implying that~$u\notin V(G)$.

We design an acyclic directed graph~$D$ defined on the vertex set~$V(D_v) 
\dunion \{v\} \dunion \{v_1,\ldots,v_{\defect(v)}\}$ such that~$G \dunion t
 K_1 \dunion \defect(v) K_1$ is the competition graph of~$D$. Its 
existence implies~$\cp(G) \leq t+ \defect(v)$ and thus the theorem. 

To design~$D$, we start with~$D_v$. To the vertex set of~$D_v$, we add the 
set~$\{v,v_1,v_2,\ldots,v_{\defect(v)}\}$.
Let~$C$ be a largest clique 
in~$G[N(v)]$ and let~$\{z_1,\ldots,z_{\defect(v)}\} := N(v)\setminus C$ be the
set of neighbors of~$v$ not contained in~$C$.

We first describe the set of arcs of~$D$ informally. We redirect all 
arcs that previously ended in~$u$ to now end in~$v$.
Next we add all arcs~$(x,u)$ with~$x\in C$ and the arc~$(v,u)$. This is possible since~$u \notin \{v\} \cup C$.
For each~$i\in \{1,\ldots,\defect(v)\}$ we add the arcs~$(v,v_i)$ 
and~$(z_i,v_i)$. 

More formally,~$D$ is the directed graph on the vertex set~$V(D_v) \dunion \{v\} 
\dunion \{v_1,\ldots,v_{\defect(v)}\}$. Its set of arcs is formed as follows:
We set
\begin{eqnarray*}
E_1 & :=& \{(y,u) \mid y \in N_{D_v}(u)\}\\
E_2 &:= & \{(y,v) \mid y \in N_{D_v}(u)\} \\
E_3 &:= & \{(x,u) \mid x\in C\} \cup \{(v,u)\}\\
E_4 &:= & \{(v,v_i),(z_i,v_i) \mid i\in \{1,\ldots,\defect(v) \}\},
\end{eqnarray*}
and we define the arcs of~$D$ as~$E(D) := (E(D_v) \setminus 
E_1) \cup E_2 \cup E_3\cup E_4$.

The directed graph~$D$ is acyclic: Prior to our first modification of the arcs,~$v$ is 
an isolated vertex and~$u$ has no outgoing arcs. Thus our first modification,
i.e., removing the set~$E_1$ and adding the set~$E_2$,
redirects the arcs ending in~$u$ to end in~$v$. This can be interpreted as 
swapping the names of~$u$ and~$v$ and thus does not introduce cycles. 
By adding the arcs sets~$E_3$ and~$E_4$, we exclusively add 
arcs that end in a vertex in~$\{u, v_1,\ldots,v_{\defect(v)}\}$. However, no 
vertex in~$\{u, v_1,\ldots,v_{\defect(v)}\}$ has outgoing arcs, and thus no 
arcs in~$E_3$ or~$E_4$ can be part of a cycle. This shows that~$D$ is acyclic.

By construction~$G \dunion t  K_1 
\dunion \defect(v) K_1$ 
is the 
competition graph of~$D$ and thus we conclude that~$G$ has competition number 
at most~$t+\defect(v)$, which concludes the theorem.
\end{proof}

Lemma~\ref{lem:comp:nr:vs:defect} and Corollary~\ref{all:good:corollary} allow
us to prove Kim's conjecture~\cite{MR2176262} concerning the relation 
between the competition number and the number of holes in a graph.

\begin{theorem}\label{thm:kim:conj}
If a graph has at most~$k$ holes, then it has competition number at most~$k+1$.
\end{theorem}
\begin{proof}
We show the statement by induction on~$n$, the number of vertices of~$G$.
For~$n= 1$ the statement is obvious.
Now suppose~$n\geq 2$.
By Corollary~\ref{all:good:corollary} the graph~$G$ contains a 
good vertex~$v$. Since~$\hole(G) = \hole(G-v) + \hole(v)$, it suffices to show 
that~$\cp(G) \leq \max\{1,\cp(G-v)\} +\hole(v)$. Since~$v$ is good and 
therefore~$\defect(v) \leq \hole(v)$ the theorem follows from 
Lemma~\ref{lem:comp:nr:vs:defect}.
\end{proof}

For a graph~$G$, the \emph{cycle space} is the~$\mathbb{F}_2$ 
vector space where the vectors are those sets of edges of~$G$ which form a 
subgraph that has only vertices of even degree. Addition in this vector space 
is defined as the symmetric difference of the edge sets. Every hole is a vector
in the cycle space. For a graph~$G$ let the hole space~$\mathcal{H}(G)$ be 
the subspace of the cycle space spanned by all the holes.
For various graphs for which Theorem~\ref{thm:kim:conj} was previously known,
a stronger upper bound for the competition number in terms of the dimension of
the hole space has recently been proven~\cite{ComHoleSpace}. We extend this
stronger upper bound to all graphs.

\begin{theorem}\label{thm:dimension:hole:space}
For any graph~$G$, we have~$\cp(G)\leq \dim(\mathcal{H}(G)) +1$.
\end{theorem}
\begin{proof}
We show the statement by induction  on~$n$, the number of vertices of~$G$.
For~$n= 1$ the statement is obvious.
Now suppose~$n\geq 2$. By Lemma~\ref{lem:comp:nr:vs:defect},~$\cp(G) \leq 
\max\{1,\cp(G-v)\} + \defect(v)$ for every vertex~$v$. It thus suffices to argue
that~$ \dim(\mathcal{H}(G)) \geq \dim(\mathcal{H}(G-v)) + \defect(v)$ for 
every good vertex~$v$. Let~$C$ be a largest clique of~$N(v)$. 
Since~$C$ is maximal, for every vertex~$z \in N(v) \setminus C$ there is a
non-adjacent vertex~$u \in C$.
Since~$v$ is good, it then follows that for every vertex~$z\in N(v) \setminus 
C$ there exists a hole~$H_z$ which contains~$v$ and~$z$ and no other vertices 
from~$N(v) \setminus C$. Let~$\mathcal{B}$ be a basis for the 
hole space~$\mathcal{H}(G-v)$, and consider the set~$\mathcal{B}' := 
\mathcal{H}(G-v) \cup \{ H_z \mid z\in N(v) \setminus C\}$. Every hole~$H_z$ 
has an edge, namely~$\{v,z\}$, that is not contained in any other hole 
in~$\mathcal{B}'$. Thus~$\mathcal{B}'$ is an independent set and~$ 
\dim(\mathcal{H}(G)) \geq \dim( \mathcal{H}(G-v)) + |N(v) \setminus C | =  
\dim(\mathcal{H}(G-v) ) + \defect(v)$.
\end{proof}

\bibliographystyle{abbrv}

\bibliography{quasi_line_comp_holes}

\end{document}